\documentclass[12pt, a4paper]{article}
\usepackage{custom}

\author{Andriy Haydys\\
}
\title{The infinitesimal multiplicities and orientations of the blow-up set of the Seiberg--Witten equation with multiple spinors}
{\small \date{August 24, 2018}}

\begin{document}
\maketitle

 \begin{abstract}
  I construct multiplicies and orientations of tangent cones to any blow-up set $Z$ for the  Seiberg--Witten equation with multiple spinors. 
  This is used to prove that $Z$ determines a homology class, which is shown to be equal to the Poincar\'{e} dual of the first Chern class of the determinant line bundle.  
  I also obtain a lower bound for the 1-dimensional Hausdorff measure of $Z$. 
 \end{abstract}

\section{Introduction}

Let $M$ be a closed oriented Riemannian three-manifold.
Denote by $\slS$ the spinor bundle associated with a fixed $\Spin$--structure on $M$.
Also fix a $\U(1)$--bundle $\sL$ over $M,$ a positive integer $n \in \N$ and an $\SU(n)$--bundle $E$ together with a connection $B$.
Consider triples $(A,\Psi,\tau) \in \sA(\sL) \times \Gamma\(\Hom(E,\slS \otimes \sL)\)\times (0,\infty)$ consisting of a connection $A$ on $\sL$, an $n$--tuple of twisted spinors $\Psi$, and a positive number $\tau$ satisfying the \emph{Seiberg--Witten equation with $n$ spinors}:
\begin{equation}
  \label{Eq_nSW}
  \begin{split}
    \|\Psi\|_{L^2} &= 1, \\
    \slD_{A\otimes B} \Psi &= 0, \qand \\
    \tau^2 F_A &= \mu(\Psi).
  \end{split}  
\end{equation}
A reader can find further details about~\eqref{Eq_nSW} as well as a motivation for studying these equations in \cite{HaydysWalpuski15_CompThm_GAFA,Haydys17_G2SW_toappear}.
The main result of \cite{HaydysWalpuski15_CompThm_GAFA} states in particular, that if $(A_k,\Psi_k,\tau_k)$ is an arbitrary sequence of solutions such that $\tau_k\to 0$, then there is a closed nowhere dense subset $Z\subset M$ and a subsequence, which converges to a limit $(A,\Psi,0)$ on compact subsets of $M\setminus Z$ possibly after applying gauge transformations.
Moreover, $(A,\Psi, 0)$ is a solution of~\eqref{Eq_nSW} on $M\setminus Z$ and the function $|\Psi|$ has a H\"older-continuous extension to $M$ such that $Z=|\Psi|^{-1}(0)$.

It follows from the proof of the above mentioned result that if a sequence $(A_k,\Psi_k,\tau_k)$ of solutions of~\eqref{Eq_nSW} converges to $(A,\Psi,0)$, then the set
\begin{equation}
	\label{Eq_CurvBlowUp}
\Bigl\{ m\in M \mid \ \exists r_k\to 0\ \text{ s.t. }\ r_k\int_{B_{r_k}(m)} |F_{A_k}|^2  \to\infty \Bigr \}
\end{equation}
is contained in $Z$, where $B_r(m)$ is the geodesic ball of radius $r$ centered at $m$.
This motivates the following.
\begin{defn}
	\label{Defn_BlowUpSet}
	A closed nowhere dense set $Z\subset M$ is called a blow-up set for the Seiberg--Witten equation with multiple spinors, if there is a  solution $(A, \Psi, 0)$ of~\eqref{Eq_nSW} defined over $M\setminus Z$ such that the following holds:
	\begin{enumerate}[(i)]
		\item
		\label{It_HoelderCont}
		 $|\Psi|$ extends as a H\"older-continuous function to all of $M$ and $Z=|\Psi|^{-1}(0)$;
		\item 
		\label{It_NablaPsi}
		$\int_{M\setminus Z} |\nabla^A\Psi|^2<\infty$.
	\end{enumerate}  
\end{defn}

As explained in Remark~\ref{Rem_W12forLimits} below,~\ref{It_NablaPsi} holds automatically provided $(A,\Psi, 0)$ is a limit of the Seiberg--Witten monopoles with $\tau_k\to 0$, $\tau_k\neq 0.$
The arguments used in this manuscript require $|\Psi|$ to be continuous only; 
The H\"older continuity of $|\Psi|$ is needed to ensure that the results of~\cite{Taubes14_ZeroLoci_Arx} can be applied.
In particular, a combination of~\cite{HaydysWalpuski15_CompThm_GAFA}*{App. A} and~\cite{Taubes14_ZeroLoci_Arx}*{Thm 1.3} yields that the Hausdorff dimension of  $Z$ is at most one.
Notice also that~\ref{It_HoelderCont} can be replaced by a weaker condition, for example~\cite{Taubes14_ZeroLoci_Arx}*{(1.5)}. 

\medskip

Assume for a while that $n=2$, which simplifies the upcoming discussion somewhat. 
By~\cite{Haydys:12_GaugeTheory_jlms} (see also~\cite{HaydysWalpuski15_CompThm_GAFA}*{App.\,A}) the gauge-equivalence class $[A,\Psi, 0]$ of a solution of~\eqref{Eq_nSW} on $M\setminus Z$ can be interpreted as a $\Z/2$ harmonic spinor in the sense of~\cite{Taubes14_ZeroLoci_Arx}*{(1.3)}.
This correspondence is important for the intended applications of the Seiberg--Witten theory with multiple spinors, see for example~\cite{Haydys17_G2SW_toappear}.
However, by interpreting the limit $[A,\Psi, 0]$ as a $\Z/2$ harmonic spinor  we loose information about the background  $\Spin^c$-structure (or, in our notations, about the line bundle $\sL$).
This raises naturally the question of how to recover this piece of information.   
Another question, which naturally appears in this context, is the following: Can any $\Z/2$ harmonic spinor  appear as a limit of the Seiberg--Witten monopoles?

In this preprint I give an answer to these questions by showing that the blow-up set $Z$ can be equipped with a certain infinitesimal structure, which encodes in particular the missing piece of information about the background  $\Spin^c$-structure.  

To be more precise, it follows from the results of~\cite{Taubes14_ZeroLoci_Arx} that at each point $z\in Z$ there is a tangent cone $\cZ_*=\cup \ell_j$ consisting of finitely many rays.
I show that each ray $\ell_j$ can be equipped with a weight $\theta_*^j\in\Z_{\ge 0}$ and an orientation provided $\theta_*^j\neq 0$. 
The collection of all weights and orientations is abbreviated as $(\theta, or)$. 
For example, if $Z$ is a smooth 1-dimensional submanifold, $\theta$ is a locally constant function on $Z$, i.e., each connected component of $Z$ is equipped with a non-negative integer multiplicity.
Moreover, the components with non-vanishing multiplicities are also oriented. 

The main result of this preprint is the following theorem. 
A more precise version is stated as~\autoref{Thm_ZeroLocusPDdetL_moreprecise}.

\begin{thm}\label{Thm_ZeroLocusPDdetLB}
Let $Z$ be a blow-up set for the Seiberg--Witten equation with two spinors.
The triple $(Z,\theta,{or})$ determines a class $[Z,\theta,{or}]\in H_1(M,\Z)$. This satisfies
\begin{equation}\label{Eq_HomClasEqualsPDc1}
[Z,\theta,{or}]=\PD(c_1(L)),
\end{equation}
where $L=\sL^2$ is the determinant line bundle and $\PD$ stays for the Poincar\'{e} dual class.
\end{thm}

I would like to stress that no extra assumptions on the Riemannian metric on $M$ or the regularity of $Z$ are required in Theorem~\ref{Thm_ZeroLocusPDdetLB}.  
In particular, viewing $Z$ as a subset of $M$ only, the homology class of $Z$ may be ill defined.

An interpretation of Theorem~\ref{Thm_ZeroLocusPDdetLB} is that  there are topological restrictions on blow-up sets for the Seiberg--Witten equation with a fixed $\Spin^c$-structure.
For example, if $L$ is non-trivial, then $Z$ can not be empty. 
Although this follows immediately from Theorem~\ref{Thm_ZeroLocusPDdetLB}, this statement can be proved directly by an elementary argument. 
Indeed, assume that for some non-trivial $L$ there is a solution $(A,\Psi, 0)$ of~\eqref{Eq_nSW} such that $\Psi$ vanishes nowhere, i.e., $Z=\varnothing$. 
The equation $\mu(\Psi)=\Psi\Psi^*-\tfrac 12 |\Psi|^2=0$ implies that $\Psi$ is surjective, which in turn shows that $\Psi$ is an isomorphisms  provided $n=2$.
Hence, we obtain $\Lambda^2 E\cong \Lambda^2(\slS\otimes\sL)\cong \sL^2=L$, which shows that $L$ is trivial thus providing a contradiction.
This argument shows in fact, that $Z$ must be infinite if $L$ is non-trivial.
A generalization of this is stated in~\autoref{Thm_HausdMeasDimOfZ}, which gives a lower bound for the 1-dimensional Hausdorff measure of $Z$ in terms of $c_1(L)$. 
 

The proof of \autoref{Thm_ZeroLocusPDdetLB} is based on the following observation: Any blow-up set for the Seiberg--Witten equation with multiple spinors is a zero locus of some continuous section $s$ of $L$ (details can be found at the beginning of Section~\ref{Sect_BlowUpSet}).
This of course immediately implies the statement of \autoref{Thm_ZeroLocusPDdetLB} if $Z$ is sufficiently regular, for example smooth. 
However, a priori $Z$ does not need to be smooth. 

\medskip

The reader may wonder why should we care about infinitesimal weights and orientations, since, after all, one can imagine more conventional ways to keep track of information about the determinant line bundle. 
One reason is as follows: Since the blow-up set for the Seiberg--Witten equations is never empty provided the determinant line bundle is non-trivial,  a natural question is which structure a blow-up set can be equipped with. 
Since \eqref{Eq_CurvBlowUp} is contained in $Z$, it seems reasonable to expect $Z$ to be the support of some sort of $\delta$-function, where components of $Z$ may have different multiplicities and orientations.
The approach utilized here provides one way to formalize this.

Another reason  comes from the conjectural relation between the Seiberg--Witten monopoles and $\rG_2$ instantons~\cite{Haydys17_G2SW_toappear}.
It seems plausible that if gauge--theoretic invariants of  compact $\rG_2$ manifolds in the sense of~\cite{DonaldsonThomas:98, DonaldsonSegal:09} exist, their construction should take into account not only honest $\rG_2$ instantons, but also $\rG_2$ instatons with singularities along one dimensional subsets~\cite{Donaldson18_KaehlerGeomExceptHolon_Arx}*{Sect. 3.5}. 
To the best of author's knowledge, no direct evidence is known at present, however, by comparing with Donaldson--Thomas invariants for Calabi--Yau three-folds, it seems plausible that singular $\rG_2$ instantons should play a r\^ole indeed.  
If this is true, it seems reasonable that the Seiberg--Witten monopoles with a non-empty blow-up set may be related to singular $\rG_2$ instantons. 
If so, weights and orientations of $Z$ are likely to encode information about the singularity of the corresponding $\rG_2$ instanton. 
However, how much of this, if any, is true goes beyond the goals of this preprint. 

\medskip

This manuscript is organized as follows. 
In Section~\ref{Sect_ZeroLocus} I discuss the topology of the zero locus of a continuous section of a complex line bundle over a closed three-manifold. 
The most important point is that the zero locus is Poincar\'e dual to the first Chern class of the line bundle even if the classical transversality condition is replaced by a weaker one, see \autoref{Prop_NaturalHomPD}.
This section is self-contained and no acquaintance  with the Seiberg--Witten equations is required.    

In Section~\ref{Sect_BlowUpSet} I prove the main result of this manuscript, \autoref{Thm_ZeroLocusPDdetLB}, and discuss some consequences.
\autoref{Thm_ZisPDforLargeN} below generalizes   \autoref{Thm_ZeroLocusPDdetLB} for any $n\ge 2$.	
Notice, however, that an extra  hypothesis, which becomes vacuous for $n=2$, appears in Theorem~\ref{Thm_ZisPDforLargeN}. 
Examples of solutions of \eqref{Eq_nSW} with $\tau=0$ and even $n\ge 2$ satisfying this hypothesis are constructed in in Section~\ref{Sect_ConstrFueter}. 
This also yields explicit examples of Fueter sections~\cite{HaydysWalpuski15_CompThm_GAFA}*{App.~A} with values in the moduli space  $\mathring M_{1,n}$ of framed centered charge one $\SU(n)$-instantons over $\R^4$.

Let me also note in passing that many aspects discussed here have analogues for other gauge-theoretic problems such as flat $\PSL(2;\C)$-connections on three-manifolds~\cite{Taubes13_PSL2Ccmpt}. 
This will be discussed in detail elsewhere~\cite{Haydys_SWflatPSL2R_InPrep}.

In a slightly different direction, presumably much of the material  of Section~\ref{Sect_ZeroLocus} can be generalized to higher dimensions. 
This in turn may be of interest for the analysis of blow up sets of some gauge-theoretic equations in dimension four such as the Seiberg--Witten equations with multiple spinors~\cite{Taubes16_SWDim4_Arx}, Kapustin--Witten equations~\cite{Taubes18_SequencesKapustWitten_Arx}, $\SL(2;\C)$ anti-self-duality equations~\cite{Taubes13_CxASD_Arx}, and Vafa--Witten equations~\cite{Taubes17_VafaWitten_Arx}.

\medskip

\textsc{Acknowledgement.} I am grateful to M.\;B\"ockstedt, M.\;Callies, A.\;Doan, and S.\;Goette for helpful discussions and also to an anonymous referee for useful suggestions and comments. 
This research was partially supported by Simons Collaboration on Special Holonomy in Geometry, Analysis, and Physics.

\section{The zero locus of a continuous section of a line bundle}
\label{Sect_ZeroLocus}

It is a classical fact that if a smooth section $s$ of a line bundle $L$ intersects the zero locus transversely, than $s^{-1}(0)$ is a smooth oriented embedded submanifold whose homology class is the Poincar\'{e} dual to the first Chern class of $L$.
Even though a \emph{generic} section does intersect the zero section transversely, in applications one can not always assume that a section at hand is in fact generic. 
This may happen for instance when $s$ satisfies some sort of PDE and perturbations are not readily available or do not fit the set-up.

In this section the structure and topology of $s^{-1}(0)$ is studied in the case when $s$ is assumed to be continuous only (and the base manifold $M$ is three-dimensional).
In particular, no transversality arguments are available and $s^{-1}(0)$ is allowed to have singularities.

\subsection{The case of an embedded graph}

In this subsection I prove that if the zero locus $Z$ of a continuous section of a complex line bundle $L$ is an embedded graph, then edges can be equipped with weights and orientations resulting in a singular $1$-chain, whose homology class represents $\PD(c_1(L))$.  
Although the material of this section is elementary, this provides a useful toy model for what follows in the next section.

\medskip

Let $G$ be an (abstract) graph with a finite set of edges $E$. 

\begin{defn}[\cite{GaroufLoebl06_NonCommJonesFn}*{Def.\,2.2}]
	\label{Defn_FlowOnGraph}
	A \emph{flow} on $G$ consists of a weight function $\theta\colon E\to\Z_{\ge 0}$ and orientations of edges with non-zero weights such that for each vertex $v$ we have
	\begin{equation}\label{Eq_Flow}
	\sum\limits_{e \text{ begins at } v}\theta(e)= \sum\limits_{e \text{ ends at } v}\theta(e),
	\end{equation}	
\end{defn}

Notice that there is one minor difference between the above definition and the one of ~\cite{GaroufLoebl06_NonCommJonesFn};
Namely, the weights in~\cite{GaroufLoebl06_NonCommJonesFn} are assumed to be positive whereas here $\theta$ is allowed to attain the zero value. 

Observe that  the set $\Flow (G)$ of all flows on $G$ has a natural structure of an abelian group.
Indeed, declare the sum of two flows $(\theta_1, or_1)$ and $(\theta_2, or_2)$ to be $(\theta, or)$ according to the following rule:
\begin{itemize}
	\item If $e\in E$ has the same orientation \wrt $or_1$ and $or_2$, then this is also the orientation of $e$ \wrt  $or$ and $\theta(e)=\theta_1(e)+\theta_2(e)$;
	\item If the orientations of $e$ \wrt $or_1$ and $or_2$ are opposite or one of the weights vanishes, declare $\theta(e):=|\theta_1(e)-\theta_2(e)|$;  If $\theta(e)\neq 0$, declare the orientation of $e$ to be the one corresponding to the bigger weight. 
\end{itemize}
Clearly, the inverse element is obtained by reversing the orientations of all edges.

Equivalently, we can also choose arbitrarily a reference orientation of \emph{all} edges. 
Then a flow $(\theta, or)$ on $G$ can be conveniently interpreted as a map $ \Theta\colon E\to\Z$, where the sign of $\Theta(e)$  encodes the difference between orientations of $e$ \wrt $or$ and the reference orientation. 
In particular, this shows that $\Flow (G)$ is a subgroup of the free abelian group generated by $E$; Hence, $\Flow(G)$ is also free and finitely generated.

\begin{defn}
	A subset $Z$ of a manifold $M$ is called \emph{an embedded graph} if there is a finite subset $V\subset Z$, whose elements are called vertices, such that $Z\setminus V$ consists of  finitely many connected components; 
	The closure of each component, which is called an edge, is a smooth embedded 1-dimensional submanifold, possibly with boundary.
\end{defn}

Let $E$ denote the set of all edges of an embedded graph $Z$.
Notice that by introducing extra vertices if necessary we can always assume that each edge contains at least one vertex (loops are allowed).

Let $Z$ be an embedded graph equipped with a flow. 
Assume for simplicity that the ambient manifold $M$ is compact. 
Clearly, $\sum_{e\in E}\theta(e)e$ is a singular  $1$-cycle in $M$. 
Denote
\[
[Z,\theta, or]:=\Bigl [\;  \sum\limits_{e\in E}\theta(e)e\; \Bigr ]\in H_1(M;\Z).
\] 
If confusion is unlikely to arise, we write simply   $[Z]:=[Z,\theta, or]$ for brevity.
It should be pointed out that this notation does not suggest that the corresponding class depends on $Z$ as a set only.

The following proposition summarizes the above considerations.

\begin{proposition}
	\label{Prop_HomolClassEmbGraph}
	For each embedded graph $Z$ in a compact manifold $M$ there is a natural homomorphism 
	\[
	\pushQED{\qed} 
	\Gamma\colon \Flow (Z)\To H_1(M;\Z), \qquad (\theta, or)\mapsto [Z,\theta, or].\qedhere
	\]
\end{proposition}

Let $s$ is a continuous section of a complex line bundle $L$ over a closed three-manifold $M$ such that the zero locus $Z:=s^{-1}(0)$ is an embedded graph.
Denote by $E$ the set of edges, which is finite. 
The section $s$ may be used to produce a flow on $Z$ as follows.
 Pick an edge $e\in E$ and a point $p$ on $e$ such that $p$ is not a vertex.
 Choose $\e>0$ so small that $H_1\bigl ( B_\e (p)\setminus e;\Z \bigr )\cong \Z$, where $B_\e(p)$ is a ball of radius $\e$ in a chart centered at $p$.
 Furthermore, chose an embedded circle $\gamma_e\subset B_\e(p)\setminus e$ such that $\gamma_e$ generates  $H_1\bigl ( B_\e (p)\setminus e;\Z \bigr )$ for some choice of orientation on $\gamma_e$.
 Notice that at this point $\gamma_e$ is not assumed to be equipped with an orientation.
 
 Furthermore, choose a local trivialization of $L$ over $B_\e (p)$ so that the restriction of $s$ to $\gamma_e$ can be thought of as a map $\gamma_e\to \C^*$. 
 Declare 
 \begin{equation}
	 \label{Eq_WeightForGraphs}
 \theta (e):=\bigl | \deg (s\colon \gamma_e\to \C^*)\bigr|\in\Z_{\ge 0} ,
 \end{equation}
 where $\deg$ denotes the topological degree. 
 Notice that this definition implicitly requires a choice of orientation of $\gamma_e$, however for the absolute value of the degree this choice is immaterial. 
 If $\theta (e)\neq 0$, there is a unique orientation of $\gamma_e$ such that 
 \begin{equation}
	 \label{Eq_OrientationCondition}
 \deg (s\colon \gamma_e\to \C^*)>0.
 \end{equation}
 Since the ambient manifold $M$ is oriented, the orientation of $\gamma_e$ yields a unique orientation of $e$. 
 
 Clearly, the map~\eqref{Eq_WeightForGraphs} as well as the orientations of edges with non-vanishing weights depends only on $s$ but not on the choices made in its definition.

\begin{proposition}
	\label{Prop_ZeroLocusEmbeddedGraph}
	Let $M$ be an oriented three-manifold. 
	Let $s$ be a continuous section of a complex line bundle $L\to M$ whose zero locus $Z:=s^{-1}(0)$ is an embedded graph with the finite set of edges $E$.
	Then the following holds:
	\begin{enumerate}[(i)]
		\item \label{It_FlowOnZeroLocusIfGraph}
		$(\theta, or)$ is a flow on $Z$, where $\theta$ is defined by~\eqref{Eq_WeightForGraphs} and the orientation of edges with non-zero weights is determined by~\eqref{Eq_OrientationCondition};
		\item \label{It_ZisPDc1IfGraph}
		If $M$ is also compact, then 
		\[
		\bigl [ Z,\theta, or \bigr ]=\PD (c_1(L)),
		\]
		where the right hand side of the equation denotes the Poincar\'{e} dual to the first Chern class of $L$.
	\end{enumerate}
\end{proposition}
\begin{proof}
	Pick a vertex $v\in Z$, an open contractible neighborhood  $U$ of $v$ such that $\Sigma:=\partial U$ is a smoothly embedded surface in $M$, and a trivialization of $L$ over a neighborhood of the closure $\bar U$ (shrink $U$ if necessary).
	Without loss of generality we can also assume that $\Sigma$ intersects each edge at most at one point and that this intersection is transverse. 
	In particular, $\Sigma\cap Z$ consists of finitely many points, say $m_1, \dots, m_k$.  
	For each $m_j$ choose a small embedded disc $D_j\subset \Sigma$ containing $m_j$; Clearly, these discs can be chosen so that there closures are disjoint. 
	Denote also $\gamma_j:=\partial D_j$.
	Notice that $\Sigma$ is oriented as the boundary of $U$; This in turn induces an orientation of each $D_j$ and, hence, also of $\gamma_j$.
	
	With these preliminaries at hand we have 
	\begin{equation}
	   \label{Eq_AuxSumOfLocDegrees}
	\sum_{j=1}^k \deg \bigl ( s\colon \gamma_j\to\C^* \bigr) = \deg (L|_{\Sigma})= 0, 
	\end{equation}
	where the last equality holds by the triviality of $L$ over $\Sigma$.
	
	Furthermore, notice that if an edge $e$ begins at $v$ and intersects $\Sigma$ at some $m_j$ we have $\theta(e)=+\deg \bigl ( s\colon \gamma_j\to\C^* \bigr)$, whereas if $e$ ends at $v$ we have $\theta(e)=-\deg \bigl ( s\colon \gamma_j\to\C^* \bigr)$. 
	Hence, \eqref{Eq_AuxSumOfLocDegrees}  shows that  $(\theta, or)$ is a flow, thus proving~\textit{\ref{It_FlowOnZeroLocusIfGraph}}.

	To see~\textit{\ref{It_ZisPDc1IfGraph}}, let $s_1\in C^\infty(M;L)$ be a perturbation of $s$ intersecting the zero section transversely.
	In particular, $s_1^{-1}(0)$ is a smooth closed oriented curve representing $\PD (c_1(L))$.  
	This is schematically shown on  Figure~\ref{Fig_GraphAndPerturb}.
	Notice that each ``blue'' connected component must be equipped with weight $1$.

	\begin{figure}[ht]
		\centering
		\begin{minipage}{.47\textwidth}
			\centering
			\raisebox{0\height}{\includegraphics[width=\linewidth]{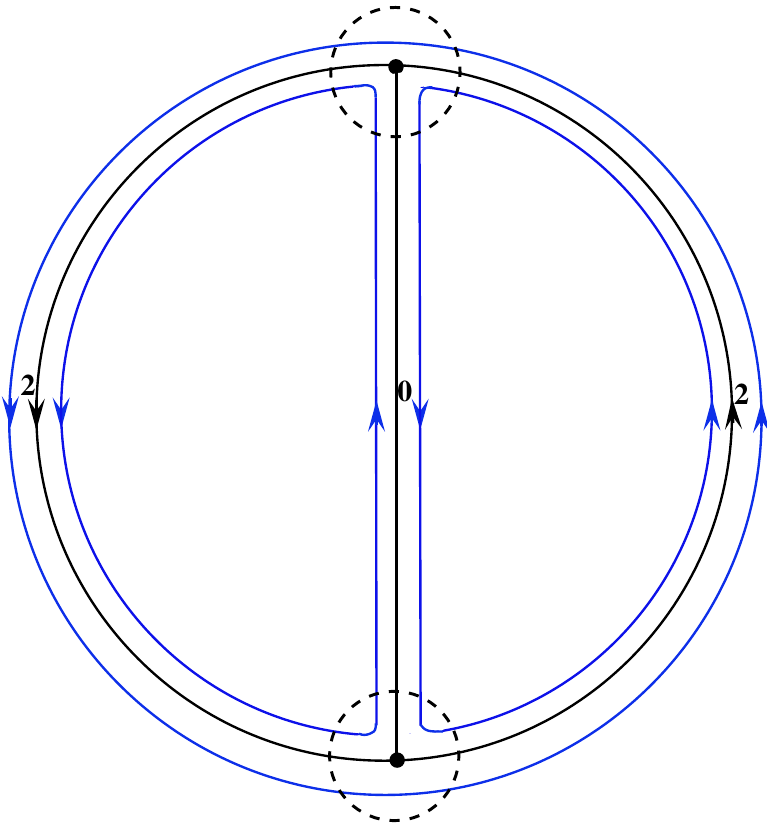}}
			\captionof{figure}{A graph (in black) and its perturbation (in blue); Numbers near black edges represent their weights; Dashed lines represent boundaries of balls to be collapsed.}
			\label{Fig_GraphAndPerturb}
		\end{minipage}%
		\hfill
		\begin{minipage}{.47\textwidth}
			\centering
			\raisebox{0.08\height}{\includegraphics[width=\linewidth]{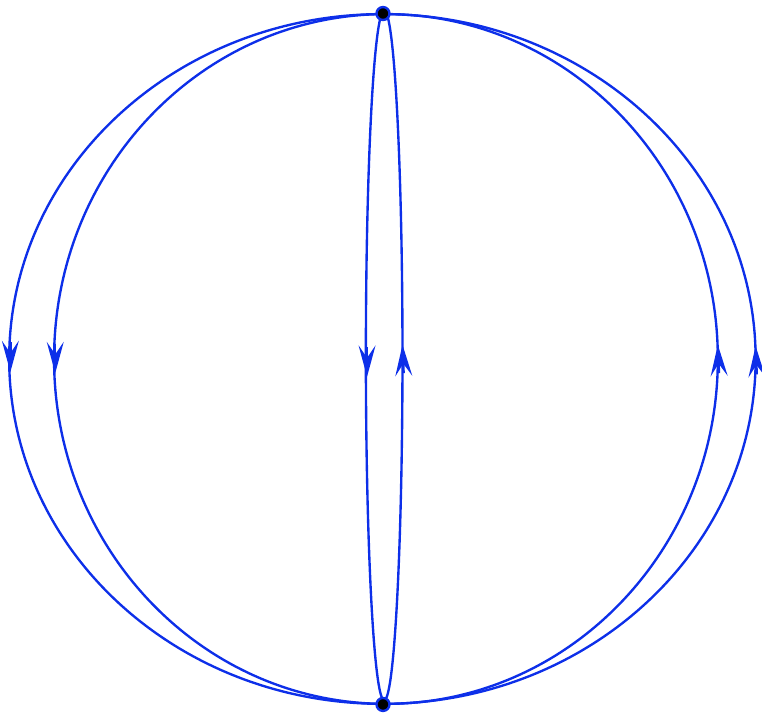}}
			\captionof{figure}{The graph obtained by a perturbation and collapse of balls. \\ }
			\label{Fig_ResultingGraph}
		\end{minipage}
	\end{figure}

	Without loss of generality we can assume that $s_1^{-1}(0)$ is contained in an arbitrarily small neighborhood of $Z$.
	By ``collapsing'' suitable disjoint balls centered  at the vertices of $Z$, we obtain a new graph $G$ shown on Figure~\ref{Fig_ResultingGraph}. 
	The graph $G$ can be though of as being obtained from $Z$ by replacing each edge $e$ by a number of ``parallel'' edges all equipped with an orientation and weighted by $1$.
	
	For any $e\in E(Z)$ denote by $E_e(G)$ the set of all edges in $G$ connecting the same vertices as $e$.
	We have
	\[  
	[Z]=\Bigl [\sum\limits_{e \in E(Z)} \theta(e)e\Bigr ]=
	\Bigl [\sum\limits_{e \in E(Z)}\sum\limits_{e' \in E_e(G)}e' \Bigr ]=
	\Bigl [\sum\limits_{e' \in E(G)}e' \Bigr ]=
	[s_1^{-1}(0)]=
	\PD(c_1(L)).
	\]
This finishes the proof of this proposition.
\end{proof}

\begin{remark}
	\label{Rem_FlowsViaCohom}
	Let $Z\subset M$ be an embedded graph. 
	Choose a neighborhood $V$ of $Z$ such that $V$ retracts on $Z$ and the boundary $\partial \overline V$ is a smoothly embedded surface. 
	One can think of $V$ as a thickening of $Z$.
	The long exact sequence of the pair $(\overline V, \partial \overline V)$ yields
	\[
	H_1(Z)\cong H_1(\partial \overline V)/ H_2(\overline V, \partial \overline V)\cong H^1(\partial \overline V)/H^1(\overline V).
	\]
	Since $\Flow(Z)\cong H_1(Z)$, the rightmost space is isomorphic to $\Flow(Z)$, which is intuitively clear, since its elements can be seen as ``assigning weights'' to elements of $\Im \bigl( H_2(\overline V, \partial \overline V)\to H_1(\partial \overline V)\bigr )$. 
\end{remark}

\subsection{The case of a graph-like set}

Let $M$ be a closed oriented three-manifold. 
It is convenient to fix a Riemannian metric $g$ on $M$ and a positive number $r_0$ which is smaller than the injectivity radius of $g$.
For any point $m\in M$ denote by $\exp_m\colon T_mM\to M$ the exponential map and $\exp_{m,\l} (\rv):=\exp_m(\l\rv)$, where $\l>0$.
These maps are defined on the balls centered at the origin and of radius $r_0$ and $\l^{-1}r_0$ respectively.  
If the point $m$ is clear from the context we will write simply $\exp$ and $\exp_\l$ respectively. 
Denote also by $d\colon M\times M\to \R_{\ge 0}$ the distance function corresponding to $g$.

Let $Z$ be a closed subset of $M$ and $\cZ_*$ be a dilation-invariant subset of $T_zM$, where $z\in Z$.
Following \cite{Taubes14_ZeroLoci_Arx}, we say that $\cZ_*$ is a rescaling limit of $Z$ at $z$, if there is a sequence of positive numbers $\l_i\to 0$  with the following property: For any $\e>0$ there is $I_\e>0$ such that for all  $i\ge I_\e$ we have:
\begin{enumerate}[(a)]
	\item \label{It_RescalinLimitA} $\exp_{\l_i}^{-1}(Z)\cap B_{\e^{-1}}(0)\subset U_\e(\cZ_*)$.
\end{enumerate}
Here $B_{\e^{-1}}(0)$ is the ball of radius $\e^{-1}$ centered at the origin and $U_\e(\cZ_*)$ is the $\e$-neighborhood of $\cZ_*$.

\begin{defn}
	A set $Z$ is said to be \emph{locally graph-like}, if the following holds:
	\begin{enumerate}[(i)]
		\item At each $z\in Z$ there is a rescaling limit, which is a cone consisting of finitely many rays;
		\item At all but at most countably many points of $Z$ there is a rescaling limit, which is a line.
	\end{enumerate}
\end{defn}

In what follows for a locally graph-like set we consider only those rescaling limits, which are cones consisting of finitely many rays.

Let $\Sigma\subset M$ be a compact embedded oriented surface.
We say that $\Sigma$ intersects $Z$ transversally, if at each point $z\in Z\cap\Sigma$ there is a rescaling limit $\cZ_*$, which is a line, such that $T_z\Sigma$ and $\cZ_*$ are transverse.
If $\cZ_*$ is equipped with a weight $\theta_*$ and an orientation, we define $I(\cZ_*, T_z\Sigma):=\epsilon\cdot \theta_*\in \Z$, where $\epsilon$ is $+1$ or $-1$ depending on the orientation of $\cZ_*\oplus T_z\Sigma$.

\begin{defn}
	\label{Defn_FlowGeneral}
	A flow $(\theta, or)$ on a locally graph-like set $Z$ is a collection of flows on each rescaling limit $\cZ_{z*}$ at each point $z\in Z$ such that the following holds: Let $\Sigma$ be any compact embedded oriented surface contained in an open contractible subset of $M$ and intersecting $Z$ transversally at each point. 
	Then there is $\delta>0$ such that for each finite covering of $\Sigma\cap Z\subset\Sigma$ by \emph{disjoint} discs $D_r(z_j)\subset \Sigma$ with $r\le\delta$ and $z_j\in Z\cap \Sigma$ we have
	\[
	\sum_{j} I(\cZ_{z_j*}, T_{z_j}\Sigma)=0
	\] 
	for any choice of rescaling limits $\cZ_{z_j*}$, which is transverse to $T_{z_j}\Sigma$ at  any $z_j\in Z\cap\Sigma$.
\end{defn}

\begin{remark}
If in the setting of the previous definition $\Sigma$ intersects $Z$ at a finite number of points, the condition is that the ``intersection number'' vanishes, i.e.
\begin{equation*}
\sum_{z\in Z\cap\Sigma} I(\cZ_{z*}, T_z\Sigma)=0.
\end{equation*}
However, I do not wish to exclude a priori the case when the intersection contains infinitely many points. 
\end{remark}

Let $Z$ be a locally graph-like set equipped with a flow.
We say that $z$ is a regular point of $Z$, if there is a neighborhood $U\subset M$ of $z$ such that $Z\cap U$ is a $C^1$-embedded submanifold of $U$.
Let $Z_{\mathrm{reg}}$ denote the set of all regular points.
Clearly, at a regular point we have $\cZ_*=T_zZ_{\mathrm{reg}}$, which is equipped with a unique weight $\theta_*$ and an orientation, provided $\theta_*\neq 0$. 
In other words, over $Z_{\mathrm{reg}}$ we can regard $\theta$ as a locally constant function with values in $\Z_{\ge 0}$, i.e., $\theta$ attaches an integer multiplicity to each connected component of $Z_{\mathrm{reg}}$. 
Moreover, each connected component with $\theta\neq 0$ is equipped with an orientation. 
With this in mind it is easy to see that in the case  $Z$ is an embedded graph, the above definition yields a flow in the sense of Definition~\ref{Defn_FlowOnGraph}. 

Clearly, the set $\Flow(Z)$ of all flows on a given $Z$ has a natural structure of an abelian group.  
I show below that there is a natural homomorphism $\Flow(Z)\to H_1(M;\Z)$.  
However, before doing this let me construct some examples.

\medskip

Assume that a locally graph-like set $Z$ is the zero locus of a continuous section of a complex line bundle $L\to M$.
Let $s_*\colon T_zM\to \C^*$ be a continuous map such that $s_*^{-1}(0)=\cZ_*$. 
\begin{defn}
A pair $(\cZ_*, s_*)$ is said to be a rescaling limit of $(Z, s)$ at $z\in Z$, if there is a sequence of positive numbers $\l_i\to 0$  with the following property: For any $\e>0$ there is $I_\e>0$ such that for all  $i\ge I_\e$ in addition to~\ref{It_RescalinLimitA} we have:
\begin{enumerate}[(b)]
	\item\label{It_RescalingLimitII} There is a sequence of trivializations of $\exp_{\l_i}^*L$ such that the $C^0$-norm of $\exp_{\l_i}^*s - s_*$ over $B_{\e^{-1}}(0)\setminus U_\e(\cZ_*)$ is less than $\e$.
	\end{enumerate}  	
\end{defn}


Since $\cZ_*$ is a cone consisting of finitely many rays, by~\autoref{Prop_ZeroLocusEmbeddedGraph},~\textit{\ref{It_FlowOnZeroLocusIfGraph}} (applied to $s_*$ in place of $s$) we obtain an infinitezimal flow $(\theta_*, or)$,  i.e., a flow  on $\cZ_*$.  
The collection of all these infinitesimal flows is abbreviated as $(\theta, or)$ and we say that $(\theta, or)$ is induced by $s$.

\begin{lem}
	\label{Lem_SectionDeterminesFlow}
	Let $s$ be a continuous section of a complex line bundle $L$. 
	Assume the zero locus $Z=s^{-1}(0)$ is locally graph-like and at each point $z\in Z$ there is a rescaling limit $(\cZ_*, s_*)$ of $(Z, s)$. 
	The collection $(\theta, or)$ induced by $s$  is a flow on $Z$.
\end{lem}
\begin{proof}
	Let $\Sigma\subset M$ be a compact embedded oriented surface in $M$ contained in an open contractible set. 
	Trivialize $L$ over $\Sigma$ so that $s$ can be thought of as a map $\Sigma\to\C$. 
	Cover $\Sigma\cap Z$ by a finite collection of disjoints discs $D_r(z_j)$. 
	If $r$ is sufficiently small, we can assume $\theta_*(z_j) = |\deg (s\colon \partial D_r(z_j)\to \C^*)|$. 
	Hence, 
	\[
	\sum_{j} \epsilon(z_j)\theta_*(z_j)=
	\sum_{j} \deg \bigl ( s\colon\partial D(z_j)\to \C^*  \bigr)=0.\qedhere
	\]
\end{proof}

\medskip

The rest of this subsection is devoted to the proof of the following result.

\begin{proposition}
	\label{Prop_NaturalHomPD}
	Let $Z$ be a compact locally graph-like set. 
	\begin{enumerate}[(i)]
		\item There is a natural homomorphism $\Gamma\colon\Flow(Z)\to H_1(M;\Z), \ (\theta, or)\mapsto [Z, \theta, or]$;
		\item 
		\label{It_PD}
		 If $(\theta, or)$ is induced by a section of a line bundle $L$, then $[Z,\theta, or]=\PD(c_1(L))$.
	\end{enumerate}
\end{proposition}

It is convenient to prove an auxiliary lemma first. 
The proof of~\autoref{Prop_NaturalHomPD} is given at the end of this subsection.

\begin{lem}
	\label{Lem_ContractingMap}
	Let $Z$ be a compact locally graph-like set. 
	Then there is a smooth map $f\colon M\to M$ homotopic to ${id_M}$ such that $G:=f(Z)$ is an embedded graph. 
	Moreover, the following holds:
	\begin{enumerate}[(i)]
		\item There is a natural homomorphism $f_*\colon\Flow(Z)\to \Flow (G)$;
		\item If $Z$ is a zero locus of some continuous section $s$ of a complex line bundle $L$, then $G$ is also the zero locus of some $\hat s\in C^0(M; L)$;
		\item
		\label{It_FlowsInducedBySections}
		 If $(\theta, or)$ is a flow induced by a section of $L$, then $f_*(\theta, or)$ is a flow induced by a section of $L$;
		\item\label{It_SmallCollapse} Given $\delta>0$, the map $f$ can be chosen so that for all $m\in M$ we have $d\bigl( m, f(m) \bigr)<\delta$.
		Besides, for any neighborhood $V$ of $Z$ the map $f$ can be chosen to be identity on $M\setminus V$.
	\end{enumerate} 
\end{lem}


\begin{proof}
The proof consists of three steps.

\setcounter{step}{0}

\begin{step}\label{Step_LocSmooth}
Assume that at $z\in Z$ there is a rescaling limit $(\cZ_*, s_*)$ such that $\cZ_*$ is a line. 
Then for any neighborhood $U$ of $z$ in $M$ there is a neighborhood $Q\subset U$ and a smooth map $f_U\simeq id$ with the following properties:
\begin{enumerate}[(i)]
\item \label{It_LocSmoothEmbCurve}$f_U(Z)\cap Q$ is a smooth embedded curve;
\item\label{It_PertOfIdentity} $f_U(z)=z$ and $f_U$ is the identity map on the complement of $U$; 
\item\label{It_DiffeoOutsideQ} The restriction of $f_U$ to $U\setminus \bar Q$ is a diffeomorphism onto its image;
\item\label{It_ZeroLocusLocally}  $f_U(Z)=s_U^{-1}(0)$ for some $s_U\in C^{0}(M; L)$.
\end{enumerate}  
Moreover, $f_U$ induces a natural homomorphism $\Flow (Z)\to \Flow (f_U(Z))$. 
\end{step}

Without loss of generality we can assume that $U$ is a coordinate chart centered at $z$.
Let $(x_1,x_2,x_3)$ be local coordinates on $U$ such that $\cZ_*$ is tangent to the $x_1$--axis.
It follows from the definition of the rescaling limit that there is a cube $Q:=(-\l,\l)^3\subset U$ such that $Z\cap Q\subset C:=(-\l,\l)\times D_{\l/2}$, where $D_{\l/2}\subset \R^2$ denotes the disc of radius ${\l/2}$ ceneterd at the origin.  
Choose a smooth function $\chi$, which vanishes on $\bar C$ and equals $1$ on the complement of a cube containing $Q$.
Then $f_U(x_1,x_2,x_3)=(x_1, \chi\cdot x_2,\chi\cdot x_3)$ is homotopic to the identity map and clearly satisfies \textit{\ref{It_LocSmoothEmbCurve}} and \textit{\ref{It_PertOfIdentity}}.
Property \textit{\ref{It_DiffeoOutsideQ}} can be checked directly for a special choice $\chi(x)=\chi_1(x_1)\chi_2(x_2^2+x_3^2)$, where $\chi_1$ and $\chi_2$ are suitable functions. 
To see  \textit{\ref{It_ZeroLocusLocally}}, notice that fixing a trivialization of $L$ over $U$, the equation $\chi\cdot s=s_U\comp f_U$ determines uniquely $s_U$, which has the required property.

Furthermore, to define the homomorphism $(f_U)_*\colon \Flow (Z)\to \Flow (f_U(Z))$ it is enough to consider the points at $Z\cap Q$. 
 If $f_U(Z)\cap Q$ is not connected, then by the construction there is a disc $D\subset Q$ such that $\partial D\subset Q\setminus (-\l,\l)\times D_{\l/2}$ and $D\cap Z =\varnothing$. 
 It is then easy to see that $(f_U)_*$ must vanish for all $z\in Z\cap Q$.
 If $f_U(Z)\cap Q$ is connected, then $(f_U)_*$ is uniquely specified by $(f_U)_*=id$ at the point $z$.

\begin{step}\label{Step_GraphLoc}
Let $B$ be a ball in $M$ centered at $v\in Z$ such that $\partial B \cap Z=\{ z_1,\dots, z_n \}$, $Z$ is smooth in a neighborhood of each $z_i$, and the intersection of $\partial B$ and $Z$ is transverse. 
Then there is a smooth map $f_B\simeq id_M$ such that 
\begin{enumerate}[(i)]
\item\label{It_C1GraphLoc} $f_B(Z)\cap B$ is an embedded graph equipped with a flow such that $v$ is an $n$-valent vertex at which \eqref{Eq_Flow} is satisfied. 
Moreover, each $z_i$ is connected with $v$ by a unique edge;
\item\label{It_FBID} $f_B$ is the identity map on the complement of $B$;  
\item \label{It_ZeroLocusLoc}  $f_B(Z)=s_B^{-1}(0)$ for some $s_B\in C^{0}(M; L)$;
\item  There is a natural homomorphism $(f_B)_*\colon\Flow(Z)\to \Flow (f_B(Z))$.
\end{enumerate} 
\end{step}

Without loss of generality we can assume that the radius of $B$ equals $1$.
Choose $\e>0$ so small that $Z\cap B\setminus B_{1-2\e}(v)$ consists of $n$ smooth connected curves.
Choose also a smooth monotone function $\chi\colon [0,1]\to [0,1]$ such that
\[ 
\chi (t)=
  \begin{cases}
  0 &\text{if } t\in [0, 1-2\e],\\
  1 &\text{if } t\in [1-\e , 1]. 
  \end{cases} 
\]  
Define 
\[
f_B(x):=
	\begin{cases}
	\chi(|x|)x &\text{if } x\in B,\\
	x & \text{otherwise}.
	\end{cases}
\] 
Clearly, $f_B(Z)\cap B$ is an embedded graph, which inherits a weight function and orientation from $Z\cap B\setminus B_{1-\e}$. 
By Definition~\ref{Defn_FlowGeneral},  we have $\sum \theta(z_i)\e(z_i)=0$, which immediately implies that  \eqref{Eq_Flow} holds at $v$. 
Part \textit{\ref{It_ZeroLocusLoc}} is proved just like the corresponding statement in  Step~\ref{Step_LocSmooth}. 
Finally, the last part follows, since the restriction of $f_B$ to $B\setminus B_{1-\e}$ is a diffeomorphism onto its image.

\begin{step}
I prove the statement of this lemma.
\end{step}

Let $z$ be an arbitrary point in $Z$ and let $\cZ_*$ be a rescaling limit of $Z$ such that $\cZ_*$ consists of finitely many rays. 
By the definition of the rescaling limit, for any $\e>0$ we can find some $\l>0$ such that $\exp_\l^{-1}(Z)\cap B_1(0)\subset U_{\e}(\cZ_*)$.
Choose $r\in [\tfrac 12,1]$ such that $\partial B_{\l r}(z)\cap Z$ contains only points admitting a line as a rescaling limit.  
The existence of $r$ follows from the observation that there are uncountably many choices for $r$, however $Z$ contains at most countably many points which do not admit a line as a rescaling limit.
Denote by $B(z)$ the chosen ball and by $U(z)\subset B(z)$ the corresponding open neighborhood of a cone containing $Z\cap B(z)$.

Since $Z$ is compact, there is a finite collection of balls as above $\{ B_i=B(z_i)\mid 1\le i\le I \}$ covering $Z$.
Pick one of these balls, say $B_1$. 
For each $z\in Z\cap\partial B_1$ choose a ball $B'(z)$ such that $\diam B'(z)\le \frac 14\diam B_1$ and denote by $Q(z)$ the open subset supplied by Step~\ref{Step_LocSmooth}.
Choose a finite collection $\{ (B_k', Q_k)\mid 1\le k\le K \}$ such that $\{ Q_k \}$ covers $Z\cap\partial B_1$. 
Without loss of generality we can assume in addition that the following holds: 
If for some $i\in\{ 1,\dots, I \}$ we have $B_k'\cap U_i\neq \varnothing$, then $B_k'\subset U_i$.
Notice also, that for each $B_k'$ there is a geodesic segment through the center of $B_k'$ such that $Z\cap B_k'$ is contained in a neighborhood of this geodesic segment. 

Apply Step~\ref{Step_LocSmooth} consecutively to $B_1',\dots, B_K'$ to obtain a map $f'\simeq id_M$ such that $f'(Z)$ is a smooth embedded submanifold in a neighborhood of $\partial B_1$.
Notice that the choice of the balls $B_k'$ ensures that $\{ U_i\mid i=\overline{2,I} \}$ covers $f'(Z)\setminus B_1$. 
If the intersection $\partial B_1\cap f'(Z)$ is not transverse, we can decrease slightly the radius of $B_1$ to get rid of the non-transverse intersection points.
This can be done so that $\{ U_i\mid i=\overline{2,I} \}$ still covers the part of $f'(Z)$ which is not contained in the new $B_1$.

Apply Step~\ref{Step_GraphLoc} to obtain a map $f_1\simeq id_M$ such that $f_1(Z)\cap B_1$ is an embedded graph equipped with a flow.
Repeating this procedure consecutively for all balls $B_i$ we obtain a map $f\simeq id_M$ such that $G:=f(Z)$ is an embedded graph.
Moreover, we obtain $f_*\colon\Flow (Z)\to\Flow (G)$ by composing corresponding homomorphisms at each step of the construction; 
Also, if $Z$ is the zero locus of a continuous section, so is $G$, since this property is preserved by each step of the construction.
Part~\textit{\ref{It_FlowsInducedBySections}} can be seen by tracing through the above proof and shrinking the corresponding neigborhoods chosen above if necessary.
Finally, the last part is immediate from the construction. 
\end{proof}

\begin{remark}
	\label{Rem_NdhdContrOnGraph}
	The proof of~\autoref{Lem_ContractingMap} shows that the following holds: The map $f$ in fact can be chosen so that its restriction to a neighborhood $V$ of $Z$ is a homotopy equivalence between $V$ and $G=f(Z)$.
	The only minor modification in the proof is needed at~\autoref{Step_GraphLoc}, namely an extra collapse of a neighborhood of $Z$ in $B\setminus B_{1-2\e}(v)$.  
	
	Clearly, we can assume that the boundary of $\overline V$ is a smoothly embedded surface. 
	To see this, it is enough to pick a non-negative smooth function $\varphi$ such that $\varphi^{-1}(0) =Z$ and shrink $V$ to $\varphi^{-1}\bigl ([0,\e)\bigr )$, where $\e>0$ is a sufficiently small regular value of $\varphi$. 
\end{remark}

\begin{remark}
	The fact that $Z$ can be mapped onto an embedded graph by a map homotopic to the identity map can be proved in a less technical manner. 
	Namely, one can choose a handle decomposition of $M$ and first ``push'' a suitable subset of each 3-handle to its boundary so that $Z$ will be mapped to the union of 0-, 1-, and 2-handles. 
	Using the same sort of arguments one can push $Z$ further to a 1-skeleton of $M$. 
	This argument requires $Z$ to be closed of 2-dimensional Hausdorff measure zero only, however the resulting map will not satisfy~\textit{\ref{It_SmallCollapse}} of~\autoref{Lem_ContractingMap}, since the construction requires collapses of large subsets of $M$.
	In contrast, the map constructed in \autoref{Lem_ContractingMap} contracts a small portion of $M$ only.
	This is used in the proof \autoref{Prop_HausdMeasDimOfZ} to obtain a lower bound on the Hausdorff measure of $Z$.  
\end{remark}

\begin{proof}[\textbf{Proof of~\autoref{Prop_NaturalHomPD}}]
	Fix a neighborhood $V$ and a graph $G$ as in Remark~\ref{Rem_NdhdContrOnGraph} (for some map $M\to M$ supplied by~\autoref{Lem_ContractingMap}). 
	Observe that the long exact sequence of the pair $(\overline V, \partial \overline V)$ yields 
	\[
	0\to H_2 (\overline V, \partial\overline V )\to H_1(\partial\overline V )\to H_1(\overline V)\to 0,
	\]
	where we used $H_2(\overline V)\cong 0$ and $H_1(\overline V, \partial\overline V)\cong H^2(\overline V )\cong 0$.
	Using  $H_1( \overline V)\cong H_1(G)$,  we obtain 
	\[
	H_1(G)\cong H_1(\partial\overline V )/ H_2 (\overline V, \partial\overline V )\cong H^1(\partial\overline V )/ H^1(\overline V),
	\]
	cf. Remark~\ref{Rem_FlowsViaCohom}.

	 Furthermore, pick any map $f$ supplied by~\autoref{Lem_ContractingMap} such that $f$ is the identity map outside of $V$. 
	 Denote by $G_f:=f(Z)$ the corresponding embedded graph.
	 Let also $U$ be an open set containing $Z$ such that $U$ is homotopy equivalent to $G_f$.	 
	   Define
	\begin{equation*}
	\Gamma(Z, \theta, or): = [G_f, f_*(\theta,  or)],
	\end{equation*}
where the brackets on the right hand side denote the homology class of an embedded graph as in~\autoref{Prop_HomolClassEmbGraph}.
The commutativity of the diagram
\smallskip
\begin{center}
	\includegraphics{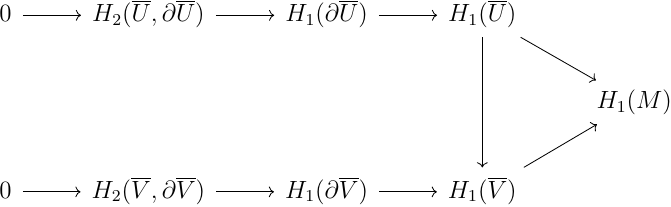}
\end{center}
together with the discussion above show that $\Gamma$ does not depend on the choice of $f$.

Part~\textit{\ref{It_PD}} of this proposition is obtained by combining~\autoref{Prop_ZeroLocusEmbeddedGraph}, \textit{\ref{It_ZisPDc1IfGraph}} and~\autoref{Lem_ContractingMap}, \textit{\ref{It_FlowsInducedBySections}}.
\end{proof}

\subsection{A lower bound for the Hausdorff measure of a locally graph-like set}

Denote by $\cH^1$ the $1$-dimensional Hausdorff measure induced by the Riemannian metric on $M$. 
\begin{proposition}
	\label{Prop_HausdMeasDimOfZ}
	Let $Z$ be a compact locally graph-like set equipped with a flow $(\theta, or)$ such that $a:=[Z,\theta, or]\neq 0$. 
	There is a positive constant $C(a, g)$ depending only on $a$ and the Riemannian metric $g$ such that
	\begin{equation}
	\label{Eq_LowerBoundH1Z_gen}
	\cH^1(Z)\ge C(a,g).
	\end{equation}
\end{proposition}
\begin{proof}
	For a subset $S\subset M$ denote by $\cH^1_\delta(S)$ the $\delta$-approximation of $\cH^1(S)$.
	For any $a\in H_1(M;\Z)$ define
	\[
	C_\delta(a,g):=\inf \bigl \{ \cH^1_\delta(G)\mid G \text{ embedded graph},\  \exists (\theta, or)\in\Flow(G),\ [G,\theta, or]=a \bigr \}.
	\]
		 
	Let $\bigl \{ V_i\mid i\in I \bigr \}$ be an arbitrary countable covering of $Z$ by open sets such that $\diam V_i<\delta$ for all $i\in I$. 
	Since $Z$ is compact, there is $\sigma>0$ such that 
	\[
	U_{\sigma}(Z):=\bigl \{ m\in M \mid d(m, Z)<\sigma \bigr \}\;\subset\; V:=\bigcup_{i\in I} V_i.
	\]
	By~\autoref{Lem_ContractingMap}, there is an embedded graph $G$ equipped with a flow $(\theta_1, or_1)$ such that $G$ is contained in $U_\sigma(Z)$ and $[G,\theta_1, or_1] =a$. 
	Since the collection $\bigl \{  V_i \bigr \}$ covers $G$, we have
	\[
	\sum_{i\in I} \diam V_i \ge \cH^1_\delta(G)\ge C_\delta(a,g).
	\]
	Hence, for all $\delta>0$ we have 
	\begin{equation}
	\label{Eq_H1ZintermBound}
	\cH^1_\delta (Z)\ge C_\delta(a,g).
	\end{equation}

	Furthermore, it follows from~\cite{Bogachev07_MeasureTheory}*{Lemma 3.10.10} that for any 1-dimensional submanifold $\gamma$ of $M$ we have $\cH^1_\delta (\gamma) = \cH^1 (\gamma)$ for any $\delta>0$ .
	Hence, for any embedded graph $G$ we have also $\cH^1_\delta (G) = \cH^1 (G)$ so that $C_\delta(a,g)$ does not actually depend on $\delta$ and equals
	\[
	C(a,g)=\inf \bigl \{  \ell(G) \mid G \text{ embedded graph},\  \exists (\theta, or)\in\Flow(G),\ [G,\theta, or]=a\bigr \},
	\]
	where 	$\ell (G):=\sum_{e\in E}\mathrm{length}(e)$ is the total length of $G$.
	
	Observe  that $C(a,g)>0$ provided $a\neq 0$. 
	Indeed,  if $C(a,g)=0$, the class $a$ could be represented by an embedded graph, whose connected components would be contained in (small) balls, which contradicts $a\neq 0$.
	
	By \eqref{Eq_H1ZintermBound}, we obtain $\cH^1_\delta(Z)\ge C(a, g)$ for all $\delta>0$ thus finishing the proof of this proposition.
\end{proof}

\subsection{The case of a rectifiable set}\label{Subsec_RectifCase}

Let $Z$ be the zero locus of $s\in C^0(M;L)$ as in the previous sections. 
Here, however, we assume the following:
\begin{enumerate}[(A)]
	\item 
	\label{Hyp_ZlocGraphlike}
	$Z$ is locally graph-like;
	\item\label{Hyp_LocHomGp} 
	For $\cH^1$--almost all $z\in Z$ there is $\bar r>0$ such that $H_1\bigl ( B_r(z)\setminus Z;\; \Z\bigr)\cong \Z$ for all $r\in (0,\bar r)$;
	\item\label{Hyp_Rectif} 
	$Z$ is a (countably) rectifiable subset of $M$ and $\cH^1(Z)<\infty$.
\end{enumerate}

Denote by $Z^\times$ the set of all those points $z\in Z$ such that at least one of the following conditions hold:
\begin{itemize}[topsep=0.2em]
	\item	$H_1\bigl ( B_{r_k}(z)\setminus Z;\; \Z\bigr)\not\cong \Z$  for some sequence $r_k\to 0$;
	\item no line is a rescaling limit of $Z$ at  $z$.
\end{itemize}
Notice that $Z^\times$ is of  $\cH^1$-measure zero.

Since $Z$ is rectifiable, we have
\begin{equation}\label{Eq_ZasUnion}
Z=\bigcup\limits_{j=0}^\infty Z_j,
\end{equation}
where $\cH^1(Z_0)=0$ and each $Z_j,\ j\ge 1$ is an embedded $C^1$-curve.
Notice that without loss of generality we can assume that \eqref{Eq_ZasUnion} is a \emph{disjoint} union \cite{Simon83_LectOnGMT}*{11.7}.

If $z\in Z_j$, then $T_zZ_j$ must be contained in any rescaling limit of $Z$ at $z$. 
Hence, if $Z$ admits a line as a rescaling limit at $z$, then this line must be $T_zZ_j$. 
If in addition $z\notin Z^\times$, then for any  $r\ll 1$ there is a circle $\gamma_z\subset B_r(z)\setminus Z$ generating $H_1\bigl (B_r(z)\setminus Z\bigr )$.  
Clearly, the multiplicity (or weight)
\[
\theta(z):=\bigl |\deg  ( s\colon \gamma_z\to \C^*) \bigr |
\]
does not depend on the choice of $\gamma_z$.
Moreover, if $\theta(z)\neq 0$, we can orient $T_zZ$ just like in the case of an embedded graph.

Thus, we obtain the multiplicity function $\theta\colon Z\setminus (Z_0\cup Z^\times)\to\Z_{\ge 0}$, which is locally constant.
This in turn determines an orientation of \[
\vec Z=\{z\in Z\setminus (Z_0\cup Z^\times)\mid \theta\neq 0 \},
\] 
which can be interpreted as  a continuous vector field $\xi$ on $\vec Z$ such that $|\xi (z)|=1$ for all $z\in\vec Z$.

\begin{proposition}\label{Prop_ZisCurrent}
If~\ref{Hyp_ZlocGraphlike}--\ref{Hyp_Rectif}  holds, then $(\vec Z, \theta,\xi)$ is an integer multiplicity current without boundary. \end{proposition} 
\begin{proof}
We only need to show that the boundary of $(\vec Z, \theta,\xi)$ is empty, i.e.,  
\begin{equation}\label{Eq_AuxNoBoundary}
\int_{\vec Z}\theta \langle df,\xi\rangle\, d\cH^1 =0
\end{equation}
for any smooth function $f$ on $M$. 
Since the left hand side of \eqref{Eq_AuxNoBoundary} is linear in $f$, it is enough to prove \eqref{Eq_AuxNoBoundary} for those functions, whose support is contained in a contractible subset.

Thus, let $U\subset M$ be contractible and $\supp f\subset U$.
Denote $f_Z= f|_{\vec Z}$.  
Let $Jf_Z$ denote the Jacobian of $f_Z$ (in the sense of the geometric measure theory).
Then we have
\begin{equation}\label{Eq_AuxZhasNoBdry}
\int\limits_{\vec Z} \theta \langle df,\xi\rangle\, d\cH^1=
\int\limits_{\vec Z} \theta \sign (\langle df,\xi\rangle  ) J f_Z\, d\cH^1= 
\int\limits_{\R}\Bigl ( \sum\limits_{z\in f^{-1}(t)\cap\vec Z} \theta(z)\e (z) \Bigr )\, d\cL^1(t),
\end{equation} 
Here the first equality follows from the definition of the Jacobian and the second one follows from the area formula.

Notice that if $t$ is a regular value of $f$, then $\Sigma_t=f^{-1}(t)\subset U$ is smooth and contained in a contractible set, namely $U$. 
Also, almost any $t\in\R$ is a regular value of both $f$ and $f_Z$ and $f(Z_0\cup Z^\times)$ is of measure zero.
Hence, an argument used in the proof of~\autoref{Lem_SectionDeterminesFlow} shows that  the right hand side of \eqref{Eq_AuxZhasNoBdry}  vanishes for almost all $t$.
\end{proof}

\begin{rem}
	As we will see below, in the case of the blow-up set for the Seiberg--Witten equation Conditions~\ref{Hyp_ZlocGraphlike} and~\ref{Hyp_Rectif} are known to hold true, whereas~\ref{Hyp_LocHomGp} requires further studies. 
	Clearly, it is also possible to replace~\ref{Hyp_LocHomGp} by other conditions, but we will not go into the details here.
\end{rem}

\section{The infinitesimal structure of the blow-up set for the Seiberg--Witten monopoles with multiple spinors}
\label{Sect_BlowUpSet}

In this section we prove~\autoref{Thm_ZeroLocusPDdetL_moreprecise}, which is a somewhat more precise version of~\autoref{Thm_ZeroLocusPDdetLB}, as well as its generalization for the case of $n\ge 3$ spinors.
Also, we obtain a lower bound for the 1-dimensional Hausdorff measure of blow-up sets for the Seiberg--Witten equations with two spinors, see~\autoref{Thm_HausdMeasDimOfZ}.

\medskip

As already mentioned in the introduction, a blow-up set for the Seiberg--Witten equations with two spinors is the zero locus of a continuous section. 
Indeed, let $(A,\Psi, 0)$ be a solution of~\eqref{Eq_nSW} over $M\setminus Z$ with $n=2$.
By~\cite{HaydysWalpuski15_CompThm_GAFA}*{Thm. 1.5} (see also Prop.\,0.1 of the Erratum) $A$ is flat with the holonomy in $\Z/2\Z$, in particular the holonomy of the induced connection on $L\to M\setminus Z$ is trivial. 
Let $s_0$ be a parallel section of $L$ over $M\setminus Z$. 
Then 
\begin{equation}
	\label{Eq_SectionS}
s:=|\Psi|\cdot s_0
\end{equation}
is a continuous section of $L$ defined on all of $M$ such that $Z=s^{-1}(0)$.

\begin{lem}
	\label{Lem_BlowUpIsLocGraphlike}
	$\phantom{A}$
\begin{enumerate}[(i)]
	\item 
	\label{It_BlowUpSetLocGraphlike}
	A blow-up set $Z$ for the Seiberg--Witten equations with $n$ spinors is a compact locally graph-like set;
	\item 
	\label{It_ZSrescLimits}
	If $n=2$, the pair $(Z, s)$ admits a rescaling limit at each point $z\in Z$ , where $s$ is given by~\eqref{Eq_SectionS}.
\end{enumerate}	
\end{lem}
\begin{proof}
Notice first  that it is enough to prove~\textit{\ref{It_BlowUpSetLocGraphlike}} for $n=2$. 
Indeed, if $n> 2$ and $Z$ is a blow-up set in the sense of Definition~\ref{Defn_BlowUpSet}, then $Z$ is also a blow-up set  for $n=2$ by Proposition~0.1 of~\cite{HaydysWalpuski15_CompThm_GAFA}.

Thus, assume that $Z$ is a blow-up set for the Seiberg--Witten equation with two spinors and let $(A,\Psi, 0)$ be a corresponding  solution of~\eqref{Eq_nSW}.
Then the projection $\psi$ of $\Psi$ is a $\Z/2\Z$-harmonic spinor~\cite{HaydysWalpuski15_CompThm_GAFA}*{Prop A.1}; Moreover, we have the pointwise equality $|\Psi|=|\psi|$ as well as the estimate $\int_{M\setminus Z}|\nabla\psi|^2<\infty$, which follows from Definition~\ref{Defn_BlowUpSet}, \textit{\ref{It_NablaPsi}}.
By~\cite{Taubes14_ZeroLoci_Arx}*{Prop. 4.1}  applied to the constant sequence $z_i=z$ and arbitrary $\l_i\to 0$ we obtain that there is a rescaling limit $(\cZ_*,\psi_*)$ of $(Z, \psi)$ at any point $z\in Z$ in the sense described by \cite{Taubes14_ZeroLoci_Arx}*{Prop. 4.1}. 
In particular, $\cZ_*$ is a cone consisting of finitely many rays \cite{Taubes14_ZeroLoci_Arx}*{Lemma 5.4} and the sequence $\exp_{\l_i}^*|\Psi|=|\Psi|\comp \exp_{\l_i}$ converges to $|\psi_*|$ in $C^0_{loc}(T_zM\setminus \cZ_*)$.
Moreover, by \cite{Taubes14_ZeroLoci_Arx}*{Lemmas 6.1 and 6.3} $\cZ_*$ is a line for all but at most countably many points of $Z$. 
In particular, $Z$ is a locally graph-like set; Clearly, $Z$ is also compact.

Furthermore, pick a smooth trivialization $\sigma$ of $L$ in a neighborhood of $z$. 
Interpret $\exp_{\l_i}^*A$ as a sequence of flat connections on the product bundle, where $L$ is trivialized by $\sigma_i:=\sigma\comp\exp_{\l_i}$. 
Hence, the sequence $\exp_{\l_i}^*A$ has a subsequence, which converges in $C^\infty_{loc}(T_zM\setminus\cZ_*)$ to some flat connection $A_*$. 
In particular, a subsequence of $\exp_{\l_i}^*s_0$, which is considered as a section of the product bundle over $T_zM\setminus\cZ_*$,  converges in $C^0_{loc}(T_zM\setminus\cZ_*)$ to a parallel section of $A_*$.
Hence, $\exp_{\l_i}^* s$ also has a subsequence, which converges to some $s_*$ in $C^0_{loc} (T_zM\setminus\cZ_*)$.
The pointwise equality $|s_*|=|\psi_*|$ implies that $s_*$ vanishes precisely on $\cZ_*$.
This shows that $s_*$ is a rescaling limit of $s$. 
\end{proof}  

\begin{remark}
	\label{Rem_W12forLimits}
		Observe that for any solution $(A,\Psi,\tau)$ of~\eqref{Eq_nSW} with $\tau\neq 0$ by the Weitzenb\"ock formula we have
		\[
		\int_M |\nabla^A\Psi|^2\le \int_M |\nabla^A\Psi|^2 +\tau^{-2}\int_M |\mu(\Psi)|^2 = -\frac 14\int_M scal_g\, |\Psi|^2\le C, 
		\] 
		where $C=\max \{ -s/4, 0 \}\ge 0$  and $scal_g$ denotes the scalar curvature of the background Riemannian metric $g$ on $M$. 
		Hence, any solution $(A,\Psi, 0)$ of~\eqref{Eq_nSW} arising as a limit of some sequence $(A_k,\Psi_k,\tau_k)$ over $M\setminus Z$ with $\tau_k\to 0$ satisfies Condition~\ref{It_NablaPsi} of Definition~\ref{Defn_BlowUpSet}.
		The fact that $(A_k,\Psi_k, 0)$ also satisfies Condition~\ref{It_HoelderCont} follows from~\cite{HaydysWalpuski15_CompThm_GAFA}*{Prop. 6.1}.
\end{remark}
		
\begin{thm}
	\label{Thm_HausdMeasDimOfZ}
	Let $Z$ be a blow-up set for the Seiberg--Witten equations with two spinors corresponding to the determinant line bundle $L$ with $a=\PD(c_1(L))\neq 0$.
	\begin{enumerate}[(i)]
		\item 
		\label{It_HausdMeasBlowUpSet}
		There is a positive constant $C(a, g)$ depending only on $a$ and the Riemannian metric $g$ such that
		\begin{equation}
		\label{Eq_LowerBoundH1Z}
		\cH^1(Z)\ge C(a,g).
		\end{equation}
		\item The Hausdorff dimension of $Z$ equals $1$.
	\end{enumerate}
\end{thm}
\begin{proof}
	Part~\textit{\ref{It_HausdMeasBlowUpSet}} follows from~\autoref{Lem_BlowUpIsLocGraphlike} and~\autoref{Prop_HausdMeasDimOfZ}. 
	Moreover,  \eqref{Eq_LowerBoundH1Z} shows in particular that the Hausdorff dimension $\dim_H Z$ of $Z$ is at least $1$.
	Combining this with $\dim_H Z\le 1$~\cite{Taubes14_ZeroLoci_Arx}*{Thm.1.3}, yields $\dim_H Z=1$.
\end{proof}

The following theorem follows directly  from~\autoref{Prop_NaturalHomPD} and~\autoref{Lem_BlowUpIsLocGraphlike}.

\begin{thm}
	\label{Thm_ZeroLocusPDdetL_moreprecise}
	Let $Z$ be a blow-up set for the Seiberg--Witten equations with two spinors. 
	For any solution $(A,\Psi, 0)$ of~\eqref{Eq_nSW} over $M\setminus Z$ define $s\in C^0(M; L)$ by~\eqref{Eq_SectionS}.
	Let $(\theta, or)$ be a flow on $Z$ induced by $s$.
	Then
	\[
	[Z,\theta, or] = \PD \bigl (c_1(L)\bigr ),
	\]
	where $[Z,\theta, or] = \Gamma (Z,\theta, or)$, $L=\sL^2$ is the determinant line bundle, and $\PD$ stays for the Poincar\'{e} dual class.\qed
\end{thm}

Theorem~\ref{Thm_ZeroLocusPDdetLB} implies  that there are restrictions for $\Z/2\Z$ harmonic spinors, which can be lifted to a solution of~\eqref{Eq_nSW} with $\tau=0$.
Namely, let $Z$ be an arbitrary locally graph-like subset of 
$M$.
By applying a map $f\simeq id_M$ if necessary, we can assume that $Z$ is an embedded graph.
Denote
\[
\Lambda=\Lambda(Z):=\Im \bigl ( \Gamma\colon\Flow (Z)\to H_1(M;\Z)  \bigr) .
\]
For instance, if $Z$ is a smooth connected oriented curve, then $\Lambda(Z)=\Z [Z]$.

\begin{proposition}\label{Prop_TopRestrBlowUpSet}
Let $(\psi, Z)$ be a $\Z/2$-harmonic spinor. 
If\ \, $\PD(c_1(L))\notin\Lambda(Z)$, then $(\psi, Z)$ can not appear as the limit of a sequence of the Seiberg--Witten monopoles with two spinors for any $\Spin^c$-structure, whose determinant line bundle is $L$.\qed
\end{proposition}

\begin{example}
To obtain an example of a $\Z/2$-harmonic spinor with a non-trivial subgroup $\Lambda$, consider a harmonic spinor  $\psi$ on a Riemann surface $\Sigma$ with a non-empty zero locus, which is necessarily a finite collection of points.
Viewing $\psi$ as a harmonic spinor on $M=\Sigma\times S^1$ equipped with the product metric, we obtain that the corresponding zero locus consists of finitely many copies of $\{pt \}\times S^1$,
 which implies that $\Lambda(\psi)= H_1(S^1)\subset H_1(\Sigma)\oplus H_1(S^1)=H_1(M)$. 

In particular, this yields the following: Choose any $\Spin^c$-structure, whose determinant line bundle is not the pull-back of a line bundle on $\Sigma$.
Then $\psi$ can not appear as the limit of a sequence of the Seiberg--Witten monopoles on $\Sigma\times S^1$ with two spinors for this choice of a $\Spin^c$-structure.   
\end{example}

\medskip

Let us turn to the general case, i.e., $n\ge 2$. 
Notice first that $E$ admits a topological trivialization, since the structure group of $E$ is $\SU(n)$ and the base manifold is three-dimensional.
It is convenient to pick such a trivialization thus identifying $E$ with the product bundle $\underline\C^n$.
Notice that $E$ may be equipped with a nontrivial background connection $B$, however this will not have any significance for the upcoming discussion.

Recall that the quadratic map $\mu$ appearing in \eqref{Eq_nSW} is obtained from the following algebraic map 
\[
\mu\colon \Hom ( {\C^{n}}, {\C^2})\to i\su(2), \qquad 
\mu(B)= BB^*-\frac 12|B|^2,
\]
which is denoted by the same letter. 
The $\U(n)$-action $A\cdot B=BA^*$ on the domain of $\mu$ combined with the action of $\R_{>0}$ by dilations yields a transitive action on $\mu^{-1}(0)\setminus \{ 0 \}$.
Observing that the stabilizer of a point, say the projection onto the first two components, is $\U(n-2)$, we obtain  
$\mu^{-1}(0)\setminus\{ 0 \}\cong \R_{>0}\times \frac {\U(n)}{\U(n-2)}$.

Furthermore, letting $\U(1)$ act as the center of $\U(n)$, we obtain a diffeomorphism 
\[
\mathring M_{1,n}=\mu^{-1}(0)\setminus\{ 0\} /\U(1)\cong \R_{>0}\times \frac{\U(n)}{U(1)\times U(n-2)}.
\] 
This yields a projection 
\begin{equation}\label{Eq_ProjM1nToGr}
\zeta\colon\mathring M_{1,n}\to\Gr_{n-2}(\C^n),
\end{equation}
which in fact represents $\mathring M_{1,n}$ as the total space of a fiber bundle over $\Gr_{n-2}(\C^n)$ with the fiber $\C^2\setminus 0/\pm 1$. 
Notice also that by viewing $\C^2$ as the tautological  $\SU(2)$-representation, we obtain an action of $\SU(2)$  on $\mathring M_{1,n}$. 
Moreover, the induced action on $\Gr_{n-2}(\C^n)$ is trivial.

Denote
\begin{equation}
	\label{Eq_Bundle_frakM}
\fM:=\SU(\slS)\times_{\SU(2)} \mathring M_{1,n}\to M.
\end{equation}
To each $\fI\in \Gamma(\fM)$ defined on a subset of $M,$ say $M\setminus Z$, with the help of \eqref{Eq_ProjM1nToGr} we can associate a map $\Phi_0\colon M\setminus Z\to \Gr_{n-2}(\C^n)$, namely $\Phi_0=\zeta\comp \fI$.

Assume $\fI$ admits a lift $\Psi$, i.e., $\Psi$ is a section of $\Hom (\underline{\C}^n, \slS\otimes\sL)$  which vanishes nowhere on $M\setminus Z$  and satisfies $\mu\comp\Psi=0$,  $\pi\comp\Psi=\fI$, where $\pi\colon\mu^{-1}(0)\setminus \{ 0 \}\to \mathring M_{1,n}$ is the natural projection. 
Under these circumstances the map $\Phi_0$ can be described as follows.
The equation $\mu(\Psi)=0$ implies that for each $m\in M\setminus Z$
the homomorphism $\Psi_m$ is surjective and therefore 
\begin{equation}\label{Eq_AssocMapToGr2}
\Phi_0\colon M\setminus Z\to \Gr_{n-2}(\C^n), \qquad \Phi_0(m)=\ker\Psi_m
\end{equation}
is well-defined and coincides with $\zeta\comp\fI$. 
Notice that over $M\setminus Z$ we have  the short exact sequence $0\to\Phi_0^*S\to\underline\C^n\to\slS\otimes\sL\to 0$, which implies $L|_{M\setminus Z}\cong \Phi_0^*\det \chk S$, where $S$ denotes the tautological bundle over $\Gr_{n-2}(\C^n)$.
\begin{remark}
$\mathring M_{1,n}$ can be identified with the framed moduli space of centered charge one $\SU(n)$-instantons on $\R^4$, see for example \cite{DonaldsonKronheimer:90}*{Sect.\,3.3}.
Then $\fM$ can be thought of as the bundle whose fiber at $m$ is $\mathring M_{1,n}(\slS_m)$.
\end{remark}

\begin{lem}\label{Lem_SeparatingZandExtraZeros}
Let $(A,\Psi,0)$ be a solution of \eqref{Eq_nSW} over $M\setminus Z$ such that $|\Psi|$ has a continuous extension to $M$ and $Z=|\Psi|^{-1}(0)$. 
Then the following holds:
\begin{enumerate}[(i), itemsep=2pt, topsep=4pt]
\item\label{It_LtrivialOverU} 
There is a neighbourhood $U$ of $Z$ such that $L$ is trivial over $U$;
\item\label{It_IntersNumberVanishes} 
For any closed oriented surface $\Sigma\subset U\setminus Z$ we have $\langle \Phi_{0*}[\Sigma], c_1(\det \chk S)\rangle =0$.
\end{enumerate} 

\end{lem}
\begin{proof}
By \autoref{Lem_BlowUpIsLocGraphlike}~\textit{\ref{It_BlowUpSetLocGraphlike}},  $Z$ is a compact locally graph-like subset of $M$.
Hence, there is $f\simeq id_M$ such that $f(Z)$ is an embedded graph. 
Then $L$ must be trivial over a regular neighborhood $V$ of $f(Z)$.
Therefore, $f^*L\cong L$ is trivial over $U:=f^{-1}(V)$. 
This proves \textit{\ref{It_LtrivialOverU}}. 
Part~\textit{\ref{It_IntersNumberVanishes}} follows immediately from  \textit{\ref{It_LtrivialOverU}} taking into account that $L|_{M\setminus Z}\cong \Phi^*\det \chk S$. 
\end{proof}

Recall \cite{HaydysWalpuski15_CompThm_GAFA}*{App.\;A} that each solution $(A,\Psi,0)$ of \eqref{Eq_nSW}  determines a  Fueter section of $\fM$ on the complement of the zero locus of $\Psi$.
Given a pair $(\fI, Z)$, where $\fI$ is a Fueter section of $\fM$ defined on the complement of $Z\subset M$, one can ask whether $(\fI, Z)$ can be obtained from a solution of \eqref{Eq_nSW} with $\tau=0$. 
This situation is considered in the following corollary.

\begin{corollary}
Let $\fI$ be a Fueter section of $\fM$ over $M\setminus Z$, where $Z=\bigsqcup Z_i$ is a link in $M$ and each $Z_i\cong S^1$. 
Let $U_i$ be a tubular neigborhood of $Z_i$. 
Denote $a_i:=\deg\bigl ( \Phi_0^*\det\chk S|_{\partial U_i}\bigr )$.
If $a(\fI, Z):=\sum_i a_i^2\neq 0$, then $\fI$ can not be lifted to a solution of \eqref{Eq_nSW} over $M\setminus Z$ with $\tau=0$. \qed
\end{corollary}

Given a link $Z$ in $M$, this corollary clearly yields restrictions on homotopy classes of sections of $\fM|_{M\setminus Z}$ representable by solutions of \eqref{Eq_nSW}.
At present, to the best of author's knowledge, the only known examples of Fueter sections with values in $\mathring M_{1,n}$ are those representable by solutions of \eqref{Eq_nSW}, see Section~\ref{Sect_ConstrFueter} below.
In particular, the question whether $a(\fI, Z)$ is a non-trivial obstruction for the existence of a solution of \eqref{Eq_nSW} with $\tau=0$ representing $(\fI, Z)$ remains open.

\medskip

In Theorem~\ref{Thm_ZisPDforLargeN} below we assume that $\Phi_0$ can be extended to all of $M$ as a continuous map. 
In particular, $\Phi_0^*\det\chk S$ is  trivial in a neighborhood of $Z$ so that obstructions based on  Lemma~\ref{Lem_SeparatingZandExtraZeros} must vanish.
In particular, for $n=2$ we have $\Gr_0(\C^2)=\{ pt \}$ so that the existence of $\Phi$ becomes trivial.

Notice that by \cite{HaydysWalpuski15_CompThm_GAFA}*{Prop.\,0.1} for any solution $(A,\Psi, 0)$ on $M\setminus Z$ we can construct a solution of the Seiberg--Witten equation with two spinors, say $(\hat A,\hat\Psi, 0)$, on $M\setminus Z$.
More precisely, $\hat\Psi$ is a homomorphism from a trivial rank two bundle into $\slS\otimes\sL\otimes \sK^{1/2}$, where $\sK=\det\bigl (\ker\Psi\bigr )^\perp\cong \Phi_0^* \det S$.
Denote by $(\theta, or)$ the flow on $Z$ associated with $(\hat A,\hat\Psi, 0)$ as in~\autoref{Thm_ZeroLocusPDdetL_moreprecise}.

\begin{thm}\label{Thm_ZisPDforLargeN}
	Let $Z$ be a blow-up set for the Seiberg--Witten equations with multiple spinors. 
Let $(A,\Psi,0)$ be a corresponding solution of \eqref{Eq_nSW} over $M\setminus Z$.
Assume the associated map $\Phi_0$ given by  \eqref{Eq_AssocMapToGr2} admits a continuous extension $\Phi\colon M\to \Gr_{n-2}(\C^n)$.
Then $(A,\Psi, 0)$ induces a flow $(\theta, or)$ on  $Z$  such that 
\[
[Z,\theta, or]=\PD(c_1(L))+ \PD(c_1(\Phi^*\det S)).
\]
\end{thm}
\begin{proof}

The hypothesis of this theorem implies that $\Phi^*\det S$ is an extension of $\sK$ to $M$; 
Then, arguing just like at the beginning of this section, we can show that $Z$ is the zero locus of a continuous section $s$ of $\sL^2\otimes\Phi^*\det S\cong L\otimes \Phi^*\det S$; 
Moreover, similar arguments to the ones used in the proof of~\autoref{Lem_BlowUpIsLocGraphlike}~\textit{\ref{It_ZSrescLimits}} show that $(Z,s)$ has rescaling limits at each point. 
The proof of this theorem is finished by appealing to~\autoref{Prop_NaturalHomPD}.  
\end{proof}

\begin{rem}
	Even though a trivialization of $E$ simplifies somewhat the discusssion above, it may be also of interest to outline necessary modifications in the case when $E$ is not trivialized. 
	For example, this may be of interest in the four-dimensional setting.
	
	Thus, denote $Q=\SU(E)\times \SU (\slS)$ the principal $G:=\SU(n)\times \SU(2)$--bundle. 
	Notice that $\SU(n)$ acts on $\mathring M_{1,n}$ by the change of frame and we can define a ``twisted version'' of~\eqref{Eq_Bundle_frakM} by
	\[
	\fM:=Q\times_G \mathring M_{1,n}\to M.
	\] 
	The map $\zeta$ is $\SU(n)$--equivariant, which implies that we have a projection
	\[
	\fM \to \Gr_{n-2}(E),
	\]
	where $\Gr_{n-2}(E)$ is the bundle over $M$ with the fiber $\Gr_{n-2}(E_m)$ at $m\in M$.
	The total space of $\Gr_{n-2}(E)$ is equipped with the ``tautological vector bundle'' $S$, whose pull-back to $\Gr_{n-2}(E_m)$ is the tautological vector bundle of this Grassman manifold.
	Under these circumstances, $\Phi_0$ is interpreted as a section of  $\Gr_{n-2}(E)\to M$.
	The rest goes through essentially without any modifications.
\end{rem}

\section{A construction of Fueter sections with values in $\mathring M_{1,2n}$}\label{Sect_ConstrFueter}

Recall that $\slS$ denotes the spinor bundle associated with a fixed \Spin-structure on $M$.
Notice that $\slS$ is equipped with a quaternionic structure, i.e., a complex antilinear automorphism $J$ s.t. $J^2=-\mathrm{id}$.

\begin{proposition}\label{Prop_FuetSectFromHarmSpinors}
Let $\psi_1,\dots,\psi_n$ be harmonic spinors. 
Denote $\Psi=(\psi_1,J\psi_1,\dots,\psi_n, J\psi_n)$.
Then $(\Psi, \vartheta, 0)$ is a solution of~\eqref{Eq_nSW} for the following data: $\sL=\underline\C$ and $\vartheta$ denotes the product connection, $E=\underline\C^n$  and $B$ is also the product connection.
\end{proposition}
\begin{proof}
We only need to show that $\mu\comp\Psi=0$. 
It is in turn enough to prove this for $n=1$. 
Indeed, if $n=1$, by the definition of $\mu$ we have  
\begin{equation}
\label{Eq_Mu2}
\mu(\psi, J\psi)= \langle J\psi,\cdot \rangle J\psi +\langle \psi,\cdot \rangle\psi - |\psi|^2.
\end{equation}
Pick a point $m\in M$ and assume without loss of generality that $\psi(m)\neq 0$.
Then $(\psi(m), J\psi(m))$ is a complex basis of the fiber $\slS_m$ and any $\varphi\in \slS_m$ can be represented as
\[
\varphi = \frac 1{|\psi(m)|^2}\bigl ( \langle J\psi(m),\varphi \rangle J\psi(m) +\langle \psi(m),\varphi \rangle\psi(m) \bigr).
\]   
Substituting this in \eqref{Eq_Mu2}, we obtain $\mu(\psi, J\psi)=0$.
\end{proof}

Combining Proposition~\ref{Prop_FuetSectFromHarmSpinors} with \cite{HaydysWalpuski15_CompThm_GAFA}*{Prop.\,A.1}, we obtain the following result.

\begin{corollary}\label{Cor_FueterSectM12n}
Any $n$-tuple of harmonic spinors $(\psi_1,\dots, \psi_n)$ defines a Fueter--section $\fI$ with values in $\mathring M_{1,2n}$ away from the common zero locus  $Z:=\psi_1^{-1}(0)\cap\dots\cap\psi_n^{-1}(0)$.
\end{corollary}

\begin{remark}
Recall that $\mathring M_{1,2}$ is isometric to $\C^2\setminus 0/\pm 1$ with its flat metric.
Hence, $\fM=\fM_{1,2}$ can be identified with $\slS\setminus 0/\pm 1$, where $0$ denotes the zero-section.
Tracing through the construction, one can see that the Fueter section of Corollary~\ref{Cor_FueterSectM12n} for $n=1$ is the $\Z/2$-harmonic spinor obtained by projecting a classical harmonic spinor.
\end{remark}

\medskip

In the remaining part of this section we construct examples of solutions of \eqref{Eq_nSW} with $\tau=0$ satisfying the hypotheses of Theorem~\ref{Thm_ZisPDforLargeN}.
To this end, pick a harmonic spinor $\psi$ and assume that its zero locus is a  smooth curve.
Choosing suitable coordinates around any fixed $m\in \psi^{-1}(0)$,  we can think of $\psi$ as a map $\R^3\to\C^2$ such that $\psi^{-1}(0)$ is the $x_3$-axis.
Assume in addition that the metric on $\R^3$ is the product metric with respect to the decomposition $\R^3=\R^2\times\R^1$.
It follows from the harmonicity of $\psi$ that along the $x_3$-axis  $\psi$ has an expansion of the form $\psi=\psi^{(N)}(x_3)w^N + o(|w|^N)$, where $w=x_1+ix_2$ and  $\psi^{(N)}(0)\neq 0$ with $1\le N<\infty$. 

Let $(\psi_1,\dots,\psi_n)$ be an $n$-tuple of harmonic spinors such that $Z:=\psi_1^{-1}(0)\cap\dots\cap\psi_n^{-1}(0)$ is a smooth curve.
By the above consideration, write $\psi_j=\psi_j^{(N_j)}(x_3)w^{N_j}+o(|w|^{N_j})$ and assume without loss of generality that $N_1=\min\{ N_j \}$.
Since
\[
\Phi_0(w,x_3)=\Bigl\{ y\in\C^n\mid \sum_{j=1}^n\bigl ( y_{2j-1}\psi_j(w,x_3)  + y_{2j}J\psi_j(w,x_3) \bigr )=0 \Bigr \}
\]
and $\bigl(\psi_1^{(N_1)}(0), J\psi_1^{(N_1)}(0)\bigr)$ is a basis of $\C^2$, the map $\Phi_0$ extends across the $x_3$-axis in a neighborhood of the origin.

Concrete examples can be constructed on $M=\Sigma\times S^1$ equipped with the product metric starting from a pair of harmonic spinors on $\Sigma$ with a common zero point.

\section{Concluding remarks}

As we have already seen in~\autoref{Lem_BlowUpIsLocGraphlike}, 
a blow-up set $Z$ for the Seiberg--Witten equation is locally graph-like. 
After the preliminary version of this preprint has been published, B.~Zhang \cite{Zhang17_Rectifiability_Arx} proved that the zero locus of a $\Z/2$ harmonic spinor is rectifiable (in dimension four).  
By~\cite{Taubes14_ZeroLoci_Arx}, $Z$ contains an open everywhere dense subset $I$, which is locally a subset of a Lipschitz graph. 
In particular, this implies that $H_1(B_r(z)\setminus Z;\; \Z)$ is either trivial or isomorphic to $\Z$ for all $z\in I$ provided $r$ is sufficiently small.
Hence, the construction of Section~\ref{Subsec_RectifCase} yields a multiplicity function $\theta$ on $I$ and an orientation on the subset where $\theta$ does not vanish.
If 
\begin{equation}
	\label{Eq_AuxH1ZI}
\cH^1(Z\setminus I)=0,
\end{equation}
 then this yields the structure of an integer multiplicity rectifiable current on $Z$.
However, at present it is not clear whether~\eqref{Eq_AuxH1ZI} (or, more generally, Condition~\ref{Hyp_LocHomGp}  on Page~\pageref{Hyp_LocHomGp}) holds.
Thus, the upshot of the above discussion is the following.
\begin{conj}
	Any blow-up set for the Seiberg--Witten equations with two spinors in dimension three admits a structure of an integer multiplicity rectifiable current, whose homology class is the Poincar\'{e} dual to the first Chern class of the determinant line bundle. 
\end{conj}

Let $Z$ and $(A,\Psi, 0)$ be as in the setting of Definition~\ref{Defn_BlowUpSet}. 
As we have already noticed at the beginning of Section~\ref{Sect_BlowUpSet}, $A$ is flat provided $n=2$. 
Hence, if $(A,\Psi, 0)$ appears as the limit of a sequence $(A_k, \Psi_k,\tau_k)$ of the Seiberg--Witten monopoles with two spinors, then the corresponding sequence of curvatures $F_{A_k}$ converges to some sort of $\delta$-function supported on $Z$. 
More precisely, it seems likely that $Z$ admits a structure of an integer multiplicity rectifiable current, such that the sequence $F_{A_k}$ converges to this current. 
While the construction described in this preprint in general does not  produce a multiplicity function and an oriention with this property (since we may change a solution $(A, \Psi, 0)$ by a gauge transformation over $M\setminus Z$ and in general the resulting structure depends  on this gauge transformation), the author believes that the choice of gauge may be fixed to obtain the desired property and intends to study this question further. 

It is also worthwhile to note that in a special case of dimensionally reduced Seiberg--Witten equations  the convergence of the sequence of curvatures to a $\delta$-current supported on $Z$ has been established in~\cite{Doan18_AdiabaticLimits}.

\bibliography{references}
\end{document}